\documentclass[11pt]{article}

\usepackage{amsmath,amsfonts,amssymb,amsthm}
\usepackage[non-compressed-cites, non-sorted-cites]{amsrefs}
\usepackage{hyperref}

\usepackage{tikz}
\usetikzlibrary{shapes}
\usetikzlibrary{arrows}

\usetikzlibrary{positioning}
\usetikzlibrary{decorations.pathreplacing}

\usepackage{enumitem}

\newenvironment{enumlist}{\begin{enumerate}[label=\upshape{(\arabic*)}]}{\end{enumerate}}

\title
{\bfseries On the existence of a \\
connected component of a graph}

\author{
Kirill Gura\thanks{Department of Mathematics, Marshall University, 1 John Marshall Drive, Huntington, WV 25755. \textit{Email:}~\texttt{kirill@gura.email}}%
\and %
Jeffry L. Hirst\thanks{Department of Mathematical Sciences, Appalachian State University,
Walker Hall, Boone, NC 28608. \textit{Email:}~\texttt{jlh@math.appstate.edu}}%
 \and %
Carl Mummert\thanks{Department of Mathematics, Marshall University, 1 John Marshall Drive, Huntington, WV 25755. \textit{Email:}~\texttt{mummertc@marshall.edu}}%
 }

\date{February 7, 2015}

\theoremstyle{plain}
\newtheorem{theorem}{Theorem}[section]
\newtheorem{corollary}[theorem]{Corollary}  
\newtheorem{lemma}[theorem]{Lemma}
\newtheorem{defn}[theorem]{Definition}      

\newcommand{\nat}{\mathbb N}  
\newcommand{\rca}{{\sf RCA}_0}
\newcommand{\aca}{{\sf ACA}_0}
\newcommand{\seq}[1]{{\langle #1\rangle}}

\newcommand{\ver}[2]{{v^{#1}_{#2}}}
\newcommand{\uver}[2]{{u^{#1}_{#2}}}

\newcommand{\cat}{{{}^\smallfrown}}

\newcommand{\proofheader}[1]{\paragraph*{\itshape\mdseries #1.}}

\newcommand{\wprob}[1]{{\mathsf{#1}}}
\newcommand{\sW}{\text{sW}}

\begin{document}

\maketitle

\begin{abstract}
We study the reverse mathematics and computability of countable graph theory, obtaining the following results. 
The principle that every countable graph has a connected component is equivalent to $\aca$ over $\rca$. The problem of decomposing a countable graph into connected components is strongly Weihrauch equivalent to the problem of finding a single component, and each is equivalent to its infinite parallelization.
For graphs with finitely many connected components, the existence of a connected
component is either provable in~$\rca$ or is equivalent to induction for $\Sigma^0_2$ formulas, depending
on the formulation of the bound on the number of components.   
\vskip\baselineskip
\textbf{MSC Subject Class (2000):} 03B30, 03D30, 03F35, 03D45
 \end{abstract}
\newpage

\section{Introduction}

The study of countable graph theory in reverse mathematics can be traced back to the founding reverse mathematics works of Friedman~\cites{fricm, frabs}, which include the graph theoretic principles K\"onig's Lemma and Weak K\"onig's Lemma.  We focus here on principles that postulate the existence of a connected component of a countable graph. 

Motivated by a suggestion of K. Hatzikiriakou,
Hirst~\cite{hirst1992} proved that the principle that a countable graph may be decomposed into its connected components is equivalent to $\aca$ over $\rca$.  In Theorem~\ref{thm:main}, we show that the principle postulating the existence of a single connected component of each countable graph is already equivalent to~$\aca$. In the subsequent section, we show that, when there is a finite bound on the number of connected components,  the strength of statements pertaining to the decomposition depends on the formulation of the finite bound.  In some cases, a connected component may be obtained in~$\rca$~(Theorem~\ref{thm:finite}), while in other cases additional induction is required~(Theorem~\ref{thm:finite2}).   

When all connected components of the graph are finite, it is natural to ask whether one can find an infinite set of vertices such that no two are in the same component. Theorem~\ref{thm:finitecomponents} shows that $\aca$ is required to find solutions to this problem. Moreover, there is a computable graph such that there is no infinite c.e. set of vertices with no two in the same component.  In Section~\ref{sec:part}, we isolate a combinatorial principle that gives another view of Theorem~\ref{thm:main}. 

In the final section, we study the problem of finding a decomposition of a graph into components and the problem of finding a single component of a graph from the viewpoint of Weihrauch reducibility. We show that these principles are strongly Weihrauch equivalent to each other and to their parallelizations, and their strong Weihrauch degree is that of~$\widehat{\mathsf{LPO}}$, which is the infinite ``parallel'' product of the limited principle of omniscience with itself. 
 
We follow Simpson~\cite{simpson} for the basic definitions and background results of reverse mathematics. This paper relies on two standard subsystems of second order arithmetic. $\rca$ includes the basic axioms $\mathsf{PA}^-$, the $\Delta^0_1$ comprehension scheme, and the $\Sigma^0_1$ induction scheme.  $\aca$ extends $\rca$ with the comprehension and induction schemes for arithmetical formulas.   We follow Hirst~\cite{hirst1992} for definitions of graph theory in reverse mathematics. 


\section{Existence of a connected component}\label{sec2}

In $\rca$, a \emph{countable graph} $G$ is a pair of sets $(V, E)$ in which the vertex set $V$ is a nonempty infinite subset of $\nat$ and the edge set $E$ is a set of unordered pairs of elements of $V$.  A (\emph{connected}) \emph{component} of a countable graph $G$ is a set $C \subseteq V$ such that:
\begin{enumlist}
\item for all $x$ and $y$ in $C$, there is a (finite) path in $G$ from $x$ to~$y$, and
\item for all $x$ and $y$ in $V$, if $x \in C$ and there is a path from $x$ to~$y$ then $y \in C$.
\end{enumlist}
$\rca$ proves that a subset of a countable graph is a connected component in this sense if and only if it is a maximal path connected subset of the graph. Our definition of a component has the advantage, compared to ``maximal path connected subset,'' of being stated as a $\Pi^0_2$ formula.   A set of vertices of a graph is \textit{totally disconnected} if no two vertices in the set are in the same connected component of the graph.
As usual, a countable graph $(V,E)$ is \emph{computable} if $V = \mathbb{N}$ and $E$ is a computable relation, and a connected component is computable if its characteristic function is computable.

Our first theorem shows that the problem of constructing even a single component of a countable graph exceeds the power of~$\rca$. 
 As usual, $\nat^{<\nat}$ denotes the set of finite sequences of natural numbers, including the empty sequence. For $\sigma, \tau \in \nat^{<\nat}$, $\sigma\cat \tau$ denotes the concatenation of $\sigma$ and $\tau$, and $\sigma \subset \tau$ indicates that $\sigma$ is a proper initial segment of~$\tau$.

\begin{theorem}\label{thm:main}
The principle that every countable graph has a connected component is equivalent to $\aca$ over~$\rca$.
\end{theorem}

\begin{proof}
First, working in $\aca$, let $G = (V,E)$ be a countable graph.
By Theorem~2.5 of Hirst~\cite{hirst1992}, there is a function $f\colon V \to \nat$ such that for all $v$ and $v^\prime$ in $V$,
$f(v) = f(v^\prime )$ if and only if $v$ and $v^\prime$ lie in the same connected component of $G$.  Let $v_0$ be
an element of $V$.  Then $\Delta^0_1$ comprehension proves the existence of the set
$C = \{ v \in V : f(v) = f(v_0 ) \}$, which is a connected component of $G$.

For the reversal, we work in $\rca$ and assume that every countable graph has a connected component.  By Lemma~III.1.3 of
Simpson~\cite{simpson}, to establish $\aca$ it suffices to prove the existence of the range of an arbitrary injection~$f\colon \nat \to \nat$.

\proofheader{Construction} We construct a countable graph~$G$. For each $\sigma \in \nat^{<\nat}$ and each $n \in \nat$, $G$ has a vertex labeled $\ver{\sigma}{n}$, and these are all the vertices.
For each $\sigma \in \nat^{<\nat}$ and each $i \in \nat$, there is an edge from $\ver{\sigma}{i}$ to $\ver{\sigma}{i+1}$.
If $f(i) = j$, then for each $\sigma \in \nat^{< \nat}$ there is an edge from $\ver{\sigma}{i}$ to $\ver{\sigma \cat \seq{j}}{0}$.
An illustration of a typical piece of the construction is shown in Figure~\ref{fig:main}.

\proofheader{Verification}
$\Delta^0_1$ comprehension suffices to form this set of edges, using $f$ as a parameter, and thus the graph $G$ can be constructed in $\rca$. By assumption, $G$ has a connected component $C$. Let $v^{\sigma}_{k}$ be a vertex of $C$; then $v^{\sigma}_0$ is also in~$C$. 

\textit{Claim:} Each $j \in \nat$ is in the range of $f$ if and only if $v^{\sigma \cat \seq{j}}_{0}$ is in $C$. The forward direction is immediate;
if $f(i) =j$ then the construction ensures that there is a path from $v^{\sigma}_0$ to $v^{\sigma\cat\seq{j}}_0$, and so $v^{\sigma\cat\seq{j}}_0 \in C$. For the converse, suppose that $v^{\sigma \cat \seq{j}}_{0}$ is in $C$.  
We can prove from the construction, using $\Pi^0_1$ induction on the length of the path, that for any path beginning at $v^{\sigma\cat\seq{j}}_0$ 
which contains no edge of the form $(v^{\sigma\cat\seq{j}}_0, v^{\sigma}_i)$, the final vertex of the path is of the form $v^{\sigma\cat\seq{j}\cat\tau}_k$ for some~$k$ and some sequence~$\tau$ which might be empty.  In particular, the final vertex of the path is not~$v^{\sigma}_0$. Thus, because there is a path from $v^{\sigma\cat\seq{j}}_0$ to $v^{\sigma}_0$ in~$C$, there must be an edge in $G$ of the form  $(v^{\sigma\cat\seq{j}}_0, v^{\sigma}_i)$. But then $f(i) = j$, so $j$ is in the range of $f$, as desired.  This proves the claim.

By the claim, the range of $f$ consists of exactly those $j$ such that $v^{\sigma \cat \seq{j}}_{0}$ is in~$C$.  Thus the range of $f$ can be formed using $\Delta^0_1$ comprehension relative to~$C$.
\end{proof}

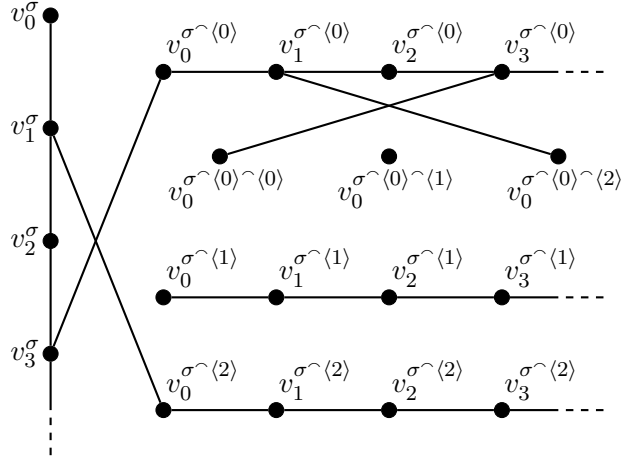
\begin{figure}
\begin{center}

\begin{tikzpicture}

\newcommand{\hgap}{1.5}
\newcommand{\vgap}{1.5}

\tikzstyle{dot}=[draw,circle,inner sep=2pt,fill]
\tikzstyle{label}=[left,shift={(0,0)}]
\tikzstyle{labelb}=[above right,shift={(-0.1,0)}]
\tikzstyle{labelc}=[below,shift={(0.1,0)}]



\foreach \name/\z in {    {$\ver{\sigma}{0}$}/0, {$\ver{\sigma}{1}$}/1, %
 				{$\ver{\sigma}{2}$}/2,%
				{$\ver{\sigma}{3}$}/3%
				}
{
	\node[dot] (A0-\z) at (0,-\vgap*\z) {};
	\node[label] at (A0-\z) {\name};
}

\foreach \from/\to in { 0/1, 1/2, 2/3}
	\draw[thick] (A0-\from)--(A0-\to);

\newcommand{\dotstart}{3.5}

\draw[thick] (A0-3)--(0,-\dotstart*\vgap);
\draw[thick] (0,-\dotstart*\vgap - 0.1)--(0,-\dotstart*\vgap - 0.2);
\draw[thick] (0,-\dotstart*\vgap - 0.3)--(0,-\dotstart*\vgap - 0.4);
\draw[thick] (0,-\dotstart*\vgap - 0.5)--(0,-\dotstart*\vgap - 0.6);



\foreach \z/\zskip in {0/0,1/2,2/3}
	\foreach \name/\x in { %
			    $\ver{\sigma\cat \seq{\z}}{0}$/0,%
			     $\ver{\sigma\cat \seq{\z}}{1}$/1, %
			     $\ver{\sigma \cat \seq{\z}}{2}$/2,%
			     $\ver{\sigma \cat \seq{\z}}{3}$/3}
	{
		\node[dot] (B\z-\x) at (\hgap+ \hgap*\x,-0.5*\vgap-\vgap*\zskip) {};
		\node[labelb] at (B\z-\x) {\name};
 	}    
	
\foreach \z in {0,1,2} 
	\foreach \from/\to in {0/1, 1/2, 2/3}
		\draw[thick] (B\z-\from)--(B\z-\to);
		
\foreach \z/\zskip in {0/0,1/2,2/3} 
{  
	\draw[thick] (B\z-3)--(4.5*\hgap,-0.5*\vgap-\vgap*\zskip); 
         \draw[thick] (4.5*\hgap + 0.1,-0.5*\vgap-\vgap*\zskip)--(4.5*\hgap + 0.2,-0.5*\vgap-\vgap*\zskip);
         \draw[thick] (4.5*\hgap + 0.3,-0.5*\vgap-\vgap*\zskip)--(4.5*\hgap + 0.4,-0.5*\vgap-\vgap*\zskip);
         \draw[thick] (4.5*\hgap + 0.5,-0.5*\vgap-\vgap*\zskip)--(4.5*\hgap + 0.6,-0.5*\vgap-\vgap*\zskip);
}



\foreach \x in {0,1,2} 
{
  \foreach \name/\z in {  $\ver{\sigma\cat \seq{0}\cat\seq{\x}}{0}$/0}
 	{
		\node[dot] (C\x-\z) at (1.5*\hgap+ 1.5* \hgap*\x,-1.25*\vgap-\vgap*\z) {};
		\node[labelc] at (C\x-\z) {\name};	
	}    
}



\draw[thick] (A0-1)--(B2-0);
\draw[thick] (A0-3)--(B0-0);
\draw[thick] (B0-3)--(C0-0);
\draw[thick] (B0-1)--(C2-0);

\end{tikzpicture}

\end{center}
\caption{Illustration of a typical piece of the construction in Theorem~\ref{thm:main}. 
The diagonal edges indicate that $f(1) = 2$ and $f(3) = 0$.}\label{fig:main}
\end{figure}

Substituting a computable enumeration of $\emptyset'$ for $f$, the construction in the proof yields a purely computability theoretic corollary.

\begin{corollary}\label{cor:main}
There is a computable graph $G$ such that every connected component of $G$ computes \textup{(}and is thus Turing equivalent to\textup{)} $\emptyset'$.
\end{corollary}

The graph in the proof of Theorem~\ref{thm:main} can be constructed with several additional properties,
allowing our results to be stated in a sharper form.   Provided that the the function~$f$ used in the construction of Theorem~\ref{thm:main} is an injection, the graph that is constructed is acyclic.  Therefore,
sharper versions of Theorem~\ref{thm:main} and Corollary~\ref{cor:main} hold, in which ``graph'' is replaced with ``acyclic graph.''

We can also characterize the connectivity of the graph. A graph is {\emph{bounded}} if there is a function $h\colon V \to V$ such that whenever $(v_1 , v_2)$ is an edge, it follows that $v_2 < h(v_1)$.  A computable graph with a computable bounding function is called {\emph{highly computable}} (or {\emph{highly recursive}}).  If the construction of Theorem~\ref{thm:main} is modified by replacing each edge from $\ver{\sigma}{i}$ to $\ver{\sigma \cat \seq{j}}{0}$ by an edge from $\ver{\sigma}{i}$ to $\ver{\sigma \cat \seq{j}}{i}$, $\rca$ proves that the resulting graph is bounded.  Therefore, a variant Theorem~\ref{thm:main} holds for countable bounded acyclic graphs, and a variant of Corollary~\ref{cor:main} holds for highly computable acyclic graphs.


\section{Graphs with finitely many components}

The graph constructed in Theorem~\ref{thm:main} has infinitely many connected components. 
In this section, we prove two results addressing graphs with finitely many connected components. 
These results show that, although decompositions of these graphs into components can be formed computably,
there are subtleties with the manner in which the bound on the number of components is stated. 

The first result applies to graphs for which there is a finite set of vertices $V_0$ such that each vertex
of $G$ is either in $V_0$ or path connected to at least one element of $V_0$.  Given such a set, $\rca$ can prove that $G$ can be decomposed into its connected components, that is, there is an $f\colon V \to \nat$ such that for all distinct vertices
$v_1$ and $v_2$, $f(v_1 ) = f(v_2 )$ if and only if $v_1$ is path connected to $v_2$.  Note that if $v_0$ is a vertex
and $f$ is such a decomposition, then $\rca$ proves the existence of the set $\{ v \in V : f(v) = f(v_0)\}$, which is
exactly the connected component containing $v_0$.

\begin{theorem}\label{thm:finite}
The following is provable in $\rca$. If $G = (V , E )$ is a countable graph and there is a finite set 
$V_0 = \{ v_0 , v_1 , \dots v_{n-1} \}$ of vertices of $G$ such that each vertex of $G$ is either in $V_0$ or path connected to at least one element of $V_0$, then $G$ can be decomposed into its connected components, and consequently $G$ has a connected component.
\end{theorem}

\begin{proof}
We work in $\rca$.  Let $G$ and $V_0$ be as described above. Form an enumeration $S_ 0 , S_1 , \ldots , S_{2^n -1}$ of the subsets of $V_0$, ordered so that larger sets come first, that is, if $S_i \subseteq S_j$ then $j\leq i$.  
Let $\varphi(k)$ denote the
formula that asserts that for every finite path $p$ in $G$, $p$ is not a path between distinct elements of $S_k$.
If $S_k$ is a singleton,
then $\varphi(k)$ holds.  Note that $\varphi(k)$ is a $\Pi^0_1$ formula.
By an easy extension of Theorem~A of Paris and Kirby~\cite{Paris} to second order arithmetic,  $\Sigma^0_1$ induction implies
the $\Pi^0_1$ least element principle.  Apply this principle to find the least $k_0$ such that $\varphi(k_0 )$ holds.  Because $S_{k_0}$ is totally disconnected and $k_0$ is minimal with this property, $S_{k_0}$ is a maximal totally disconnected subset of $V_0$, and so every vertex not in $S_{k_0}$ is path connected to exactly one vertex in~$S_{k_0}$.  Define
the function $f\colon V \to \nat$ by $f(v)=v$ if $v \in S_{k_0}$, and $f(v)$ is the unique element of $S_{k_0}$ that is path connected to $v$ otherwise.
The function $f$ exists by $\Delta^0_1$ comprehension and is a decomposition of $G$ into connected components.  By the comments
preceding the statement of the theorem, $G$ has a connected component.
\end{proof}

Our second result on graphs with finitely many connected components uses the weaker hypothesis that there is some $n$ such that every collection of $n$ vertices must include at least one path connected pair.  Under this weaker assumption, the $\Sigma^0_2$ induction scheme is required to prove the existence of a decomposition 
into connected components, or even the existence of a single component.

\begin{theorem}\label{thm:finite2}
$\rca$ proves that the following are equivalent:
\begin{enumlist}
\item ${\sf I} \Sigma^0_2$, the induction scheme for $\Sigma^0_2$ formulas with set parameters.
\item If $G = (V,E)$ is a countable graph and there is some $n$ such that every collection of $n$ vertices includes at least
one path connected pair, then $G$ can be decomposed into its connected components.
\item If $G = (V,E)$ is a countable graph and there is some $n$ such that every collection of $n$ vertices includes at least
one path connected pair, then $G$ has a connected component.
\end{enumlist}
\end{theorem}

\begin{proof}
We will work in $\rca$ throughout.  The equivalence of (1) and (2) is Theorem~4.5 of Hirst~\cite{hirst1992}.  We can capitalize on
our work above to provide a succinct alternative proof that (1) implies (2).  Suppose $G$ is as in (2).  By the $\Pi^0_2$ least element principle
(which can be deduced from (1) in $\rca$ as in Theorem~A of Paris and Kirby \cite{Paris})
 there is a least $n_0$ such that every collection of $n_0$ vertices contains a path
connected pair.  By the minimality of $n_0$, there is a set $V_0$ of $n_0 -1$ vertices no two of which are path connected.
For any vertex $v$, if $v\notin V_0$, then $V_0 \cup \{v\}$ has cardinality $n_0$, so $v$ is path connected to an element of
$V_0$.  By Theorem~\ref{thm:finite}, $G$ can be decomposed into its connected components.

The proof that (2) implies (3) follows immediately from the comments in the paragraph preceding Theorem~\ref{thm:finite}.  It
remains to show that (3) implies (1).  By the second order analog of Theorem~A of Paris and Kirby~\cite{Paris}, it suffices to use
(3) to deduce the $\Sigma^0_2$ least element principle.  Fix a $\Sigma^0_0$ formula~$\theta(k,q,s)$, which may have set parameters. Suppose that there is a $k$ for which $(\exists q)(\forall s) \theta (k,q,s)$ holds.  We will show that there is a least integer $m$ such that $(\exists q)(\forall s) \theta (m,q,s)$ holds by constructing a graph and finding a connected component.

For each $j<k$ and each $n$, let $b(n,j)$ be the least number $b$ (if any such number exists) such that:
\begin{enumerate}
\item if  $n >0$, then $b(n-1,j)$ is defined and $b > b(n-1,j)$, and 
\item there is some $s_0 < b$ such that $\neg \theta (j,n,s_0)$ holds.  
\end{enumerate}
 Informally, $b(n,j)$ is a bound for a witness that $(\forall s) \theta (j,n,s)$ fails. For each~$j$, $(\exists q)( \forall s) \theta (j,q,s)$ holds if and only if only there are only finitely many $n$ for which $b(n,j)$ is defined. 
The characteristic function for the predicate $t = b(n,j)$ is definable by primitive recursion,
so $\rca$ proves its existence.  (See Theorem~II.3.4 of Simpson~\cite{simpson}.)  Whenever $t=b(n,j)$ holds,
we have $n\le t$, so $(\exists n) [t=b(n,j)]$ is equivalent to $(\exists n \le t)[t=b(n,j)]$, a bounded quantifier applied
to a primitive recursive predicate.  We will use the formula $(\exists n) [t=b(n,j)]$
in the construction of our graph and the bounded quantifier form in verifying that $\rca$ proves the existence of the graph.
 
\proofheader{Construction}
The graph $G$ (see Figure \ref{fig:finite2})
will consist of a root vertex $\rho$ and subgraphs $G_\tau$ for each nonempty strictly decreasing sequence~$\tau$ of numbers less than~$k$. There will be $2^k -1$ such sequences. The vertices of $G_\tau$ will be
$\{ \uver{\tau}{n} : n \in \nat\}$ and $\{ \ver{\tau}{n} : n \in \nat\}$.  Next, we will define $E$, the set of edges of $G$.
For a sequence consisting of a single number $j<k$, add the following edges associated with the subgraph $G_\seq{j}$:
\begin{list}{$\bullet$}{}
\item  $(\rho , \uver{\seq{j}}{0} ) \in E$.  (This edge connects a corner of $G_\seq{j}$ to the rest of $G$.)
\item $( \uver{\seq{j}}{i} , \uver {\seq{j}}{ i+1} ) \in E$.
\item If $\neg (\exists n)[ t = b(n,j)]$ then $(\ver{\seq{j}}{t} , \ver{\seq{j}}{ t+1} ) \in E$.
\item  If $(\exists n) [ t = b(n,j)]$ then $( \uver{\seq{j}}{t}, \ver{\seq{j}}{t} ) \in E$.
\end{list}
If $\tau$ is a strictly descending sequence and $j$ is less than the last element of~$\tau$, then add the following edges for $G_{\tau\cat \seq{j}}$:
\begin{list}{$\bullet$}{}
\item  $(\ver{\tau}{ 0} , \uver {\tau \cat \seq{j} }{ 0 } ) \in E$.  (This edge connects a corner of $G_{\tau\cat \seq{j}}$ to $G_\tau$).
\item  If $(\ver {\tau }{ t} , \ver{\tau}{ t+1} ) \in E$, then put edge $(\uver {\tau \cat\seq{j}}{ t} , \uver{\tau \cat\seq{j}}{ t+1} )$ into~$E$. Also, copy the structure of $G_\seq{j}$, that is:
\begin{list}{$\bullet$}{}
\item  if $\neg (\exists n)[t = b(n,j)]$ then $(\ver{\tau \cat \seq{j} }{ t}, \ver{\tau\cat \seq{j} }{ t+1 }) \in E$, and
\item if $(\exists n ) [ t= b(n,j)]$ then $(\ver{\tau\cat \seq{j} }{ t }, \uver {\tau \cat \seq{j} }{ t} ) \in E$.
\end{list}
\item If $(\ver {\tau }{ t} , \ver{\tau}{ t+1} ) \notin E$, then
``cap'' the construction of $G_{\tau \cat \seq{j} }$ by adding $(\uver{\tau\cat \seq{j} }{t} ,\ver{\tau\cat \seq{j}}{ t}) \in E$, and
``link'' the continuation of $G_{\tau\cat \seq{j}}$ to $G_\tau$ by adding $(\ver{\tau }{ t+1 } , \uver{\tau\cat \seq{j} }{ t+1 }) \in E$. \end{list}
This completes the definition of $G$.

Figure \ref{fig:finite2} illustrates a typical construction for $G$.  In the figure, the vertical edge from $\uver{\seq{0}}{1}$ to
$\ver{\seq{0}}{1}$ in
the subgraph $G_\seq{0}$ shows that 
$0$ does not witness that $(\exists q )(\forall s) \theta (0,q,s)$, and the
similar vertical edge from $\uver{\seq{1}}{3}$ to $\ver{\seq{1}}{3}$ in $G_\seq{1}$ shows that $0$ does not
witness $(\exists q )(\forall s) \theta (1,q,s)$.  The subgraph $G_\seq{1,0}$ is a copy of $G_\seq{0}$, except at locations where $G_\seq{1}$ has
a vertical edge,  in which case a cap and a link are added.  If $(\forall s) \theta (1, 1 , s)$ and $\neg( \exists n)(\forall s )\theta (0,n,s)$, then
the component of $G$ containing the vertex $ \ver{\seq{1}}{4}$ would be disconnected from the 
component containing $\rho$, and would contain an isomorphic copy of the portion of $G_\seq{0}$
to the right of $\uver{\seq{0}}{4}$.  Based on the assumption that $\neg( \exists q)( \forall s )\theta (0,q,s)$, this copy of
the right portion of $G_\seq{0}$ would lie entirely in the same component as  $\ver{\seq{1}}{4}$.

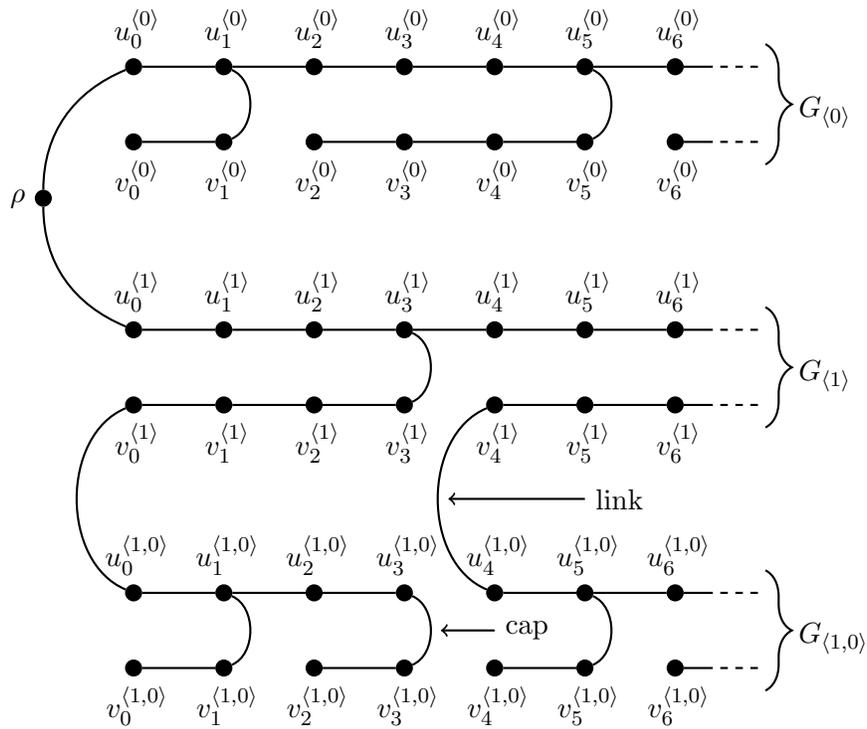
\begin{figure}
\begin{center}
\begin{tikzpicture}[x=1cm,y=1cm,thick]

\tikzstyle{dot}=[draw,circle,inner sep=2pt,fill]

%

\tikzstyle{labell}=[left,shift={(-0.1,0)}]
\tikzstyle{labelr}=[right]
\tikzstyle{labela}=[above,shift={(0.05,0.1)}]
\tikzstyle{labelb}=[below,shift={(0.05,-0.1)}]

\newcommand{\hgap}{1.2}

\newcommand{\biggap}{3.5}
\newcommand{\smallgap}{1}

\node[dot] (r) at (0,-0.5*\biggap) {};
\node[labell] at (r) {$\rho$};

\foreach \y/\l in {0/0,1/1,2/{1,0}} 
{ 

\foreach \x in {0, 1, 2, 3, 4, 5, 6} 
{
	\node[dot] (u\y-\x) at (\hgap + \hgap*\x,0- \biggap*\y) {};
	\node[labela] at (u\y-\x) {$\uver{\seq{\l}}{\x}$};
}

\foreach \x in {0, 1, 2, 3, 4, 5, 6} 
{
	\node[dot] (v\y-\x) at (\hgap + \hgap*\x,-\smallgap-\biggap*\y) {};
	\node[labelb] at (v\y-\x) {$\ver{\seq{\l}}{\x}$};
}

\foreach \f/\t in {0/1,1/2,2/3,4/5,5/6} 
{
  \draw (u\y-\f)--(u\y-\t);
}

\draw (u\y-6)--(7*\hgap+0.5,-\biggap*\y);
\draw (7*\hgap+0.6,-\biggap*\y)--(7*\hgap+0.7,-\biggap*\y);
\draw (7*\hgap +0.8,-\biggap*\y)--(7*\hgap + 0.9,-\biggap*\y);
\draw (7*\hgap + 1,-\biggap*\y)--(7*\hgap + 1.1,-\biggap*\y);

\draw (v\y-6)--(7*\hgap + 0.5,-\smallgap-\biggap*\y);
\draw (7*\hgap + 0.6,-\smallgap-\biggap*\y)--(7*\hgap + 0.7,-\smallgap-\biggap*\y);
\draw (7*\hgap + 0.8,-\smallgap-\biggap*\y)--(7*\hgap + 0.9,-\smallgap-\biggap*\y);
\draw (7*\hgap + 1,-\smallgap-\biggap*\y)--(7*\hgap + 1.1,-\smallgap-\biggap*\y);

\draw [thick,decorate,decoration={brace,mirror,amplitude=10pt}]
 (7*\hgap + 1.2,-0.3-\smallgap-\biggap*\y) -- (7*\hgap + 1.2,0.3-\biggap*\y) node[midway,xshift=8mm] {};

\node[right,yshift=-1mm] at (7*\hgap + 1.5,-0.5\smallgap - \biggap*\y) { $G_{\seq{\l}}$};

}

\draw (u0-3)--(u0-4);
\draw (u1-3)--(u1-4);

\draw (v0-0)--(v0-1);
\draw (v0-2)--(v0-3)--(v0-4)--(v0-5);

\draw (v1-0)--(v1-1)--(v1-2)--(v1-3);
\draw (v1-4)--(v1-5)--(v1-6);

\draw (v2-0)--(v2-1);
\draw (v2-2)--(v2-3);
\draw (v2-4)--(v2-5);

\draw[thick] (r) to[out=90,in=202.5] (u0-0);
\draw[thick] (r) to[out=-90,in=157.5] (u1-0);

\draw (v1-0) to[out=202.5,in=157.5] (u2-0);
\draw (v1-4) to[out=202.5,in=157.5] (u2-4);

\draw (u0-1) to[out=-22.5,in=22.5] (v0-1);
\draw (u0-5) to[out=-22.5,in=22.5] (v0-5);
\draw (u1-3) to[out=-22.5,in=22.5] (v1-3);

\draw (u2-1) to[out=-22.5,in=22.5]  (v2-1);
\draw (u2-3) to[out=-22.5,in=22.]  (v2-3);
\draw (u2-5) to[out=-22.5,in=22.]  (v2-5);

\node[right] (link) at (6*\hgap,  -1.5*\biggap + -0.5*\smallgap) {link};
\node[right] (m2) at (4*\hgap + 0.3,  -1.5*\biggap + -0.5*\smallgap) {};

\draw[->] (link)--(m2);

\node[right] (cap) at (5*\hgap,  -2*\biggap + -0.5*\smallgap) {cap};
\node[right] (m1) at (4*\hgap + 0.2,  -2*\biggap + -0.5*\smallgap) {};
\draw[thick,->] (cap)--(m1);

\end{tikzpicture}

\end{center}
\caption{Illustration of a typical piece in the construction in Theorem~\ref{thm:finite2}.}\label{fig:finite2}
\end{figure}

\proofheader{Verification}
We claim that $\rca$ proves that $G$ exists.  The set of codes for vertices is $\Delta^0_1$ definable.  In the definition
of the set of edges, the quantified formulas involving $t=b(n,j)$ are equivalent to formulas with bounded quantifiers.
Finally, inclusion of any edge with a vertex in $G_{\tau\cat \seq{j}}$ depends only on a bounded initial segment of
$G_\tau$, so $\rca$ proves the existence of the set of edges.

To verify that $G$ has finitely many connected components, we will show that any collection
of $2^k (k+3) + 1$ vertices includes at least one path connected pair.  (This is an inexact but convenient upper bound.)
First note that for each $j$, $G_{\seq{j}}$ has one connected component if $\neg (\exists q)(\forall s) \theta(j,q,s)$ and
two components if $(\exists q)(\forall s) \theta (j,q,s)$.  Thus, for each $j$, $G_{\seq{j}}$ has at most two connected components.
Suppose $W_d = \{ w_0^d , w_1^d, \dots w_n^d \}$ is a sequence of vertices in $G_{\seq{j_0, j_1, \dots j_d}}$, for some $n$, such that
no pair is connected by a path in $G$.  Because at most one connected component of $G_{\seq{j_0, j_1, \dots j_d}}$ is not
connected to $G_{\seq{j_0, j_1, \dots j_{d-1}}}$, by successively examining finite paths in $G$ we can discover the
first $n-1$ vertices of $G_{\seq{j_0, j_1, \dots j_{d-1}}}$ that are connected to distinct elements of $W_d$.
Let $W_{d-1} = \{ w_0^{d-1} , \dots , w_{n-1}^{d-1} \}$ be the sequence of these vertices.  Because no two vertices
in $W_d$ are connected by a path, no two vertices in $W_{d-1}$ are connected by a path.
$\rca$ can prove
the existence of the sequence of codes for the finite sets $W_d, W_{d-1}, \dots W_0$.  The $\Sigma^0_1$ least
element principle suffices to prove that in each of these sets, no two vertices are path connected.  Because
$W_0$ contains at most two vertices, $W_d$ contains at most $d+2$ vertices.  Generalizing, for any
strictly descending sequence $\tau$ of values less than $k$, any totally disconnected set of vertices in $G_\tau$ must
be of size at most $k+2$.  Now suppose $V_0$ is a subset of the vertices of $G$ of size $2^k (k+3)+1$.
One of these vertices may be $\rho$, but at least $2^k (k+3)$ of them lie in the $2^k -1$ subgraphs of the form $G_\tau$.
Because $\rca$ proves that the sum of less than $2^k$ numbers each less than $k+3$ is less than $2^k (k+3)$,
some single subgraph $G_\tau$ must contain at least $k+3$ vertices.  Because no such collection of $k+3$ vertices
can be totally disconnected, some pair must be connected by a path in $G$.  Thus every collection of
$2^k(k+3)+1$ vertices of $G$ contains at least one path connected pair.

Apply item (3) to find a connected component
of $G$.  Let $C$ denote the set of vertices in this component.  We will use $C$ to find the least $m$ such that
$(\exists q)(\forall s) \theta (m,q, s)$.  In the following argument, each selection of a least element is a consequence of
the $\Sigma^0_1$ least element principle.  If $\rho \in C$ and there is a $j$ such that $\ver{\seq{j}}{n} \notin C$ for some $n$,
then let $m$ be the least such~$j$.  If no such $j$ exists, let $m=k$.  If $\rho \notin C$, then there is a shortest
sequence $\tau$ and a least $n_0$ such that $\ver{\tau}{n_0} \in C$.  Because $\tau$ is decreasing,
$\min( \tau )$ is the last element of $\tau$.  If there is a $j< \min (\tau )$ such that $\ver{\tau\cat \seq{j} }{ n } \notin C$ for some~$n>n_0$,
then let~$m$ be
the least such~$j$.  If no such~$j$ exists, let~$m= \min{\tau}$.
In each of these cases, $m$ is the least integer such that $(\exists q )(\forall s )\theta (m,q,s)$, as desired.
This completes the proof that (3) implies (1) and completes the proof of the theorem. 
\end{proof}

The previous theorem extends Theorem~4.5 of Hirst~\cite{hirst1992}. In the context of classical (nonuniform) computability theory, however, there is no issue with induction, and Corollary 4.3 of Hirst~\cite{hirst1992} shows that each computable graph with finitely many connected components
has a computable decomposition into its connected components, using either formulation of the bound on the number of components.  Theorem~\ref{wr5} shows that situation is more complicated in the context of Weihrauch reducibility. 


\section{Graphs in which every component is finite}\label{section4}

In this section, we turn to a problem that is, in a certain sense, dual to the problem of the previous section. Graphs in which there are only finitely many components must have at least one infinite component. We now consider graphs in which every component is finite.  For such graphs,
the existence of a single component is trivial, so we turn to principles that imply the existence of
infinitely many connected components. There are several seemingly natural ways to state such principles. We show that three are equivalent to each other, and to $\aca$, over $\rca$. Recall that the \textit{canonical index} for a finite set $\{a_0 < a_1 < \cdots < a_k\}$ is the number $\prod_{i \leq k} p_i^{a_i}$, where $p_i$ is the $i$th prime number.

\begin{defn} The following principles are defined in $\mathsf{RCA}_0$.
\begin{itemize}
\item $\mathsf{FC}$-1: Every infinite graph in which every component is finite has a sequence $\langle d_1, d_2, \ldots\rangle$ of canonical indices of different components.
\item $\mathsf{FC}$-2: Every infinite graph in which every component is finite has a family $\langle C_i : i \in \mathbb{N}\rangle$ of distinct connected components. The family is given as a function $C\colon \mathbb{N}^2 \to \{0,1\}$.
\item $\mathsf{FC}$-3:  Every infinite graph in which every component is finite has an infinite totally disconnected set.  
\end{itemize}
\end{defn}
Principle $\mathsf{FC}$-3 can be stated equivalently ``every infinite graph has either an infinite component or an infinite totally disconnected set.'' 

\begin{lemma}\label{lem:fcforward} The following implications hold over $\mathsf{RCA}_0$:
\[
\aca \Rightarrow \mathsf{FC}\text{-}1 \Rightarrow \mathsf{FC}\text{-}2 \Rightarrow \mathsf{FC}\text{-}3.
\]
\end{lemma}
\begin{proof}
Working in $\mathsf{ACA}_0$, it is straightforward to prove $\mathsf{FC}$-1 by first decomposing the graph into components and then forming the sequence of canonical indices of the components. 

The implication of $\mathsf{FC}$-1 to $\mathsf{FC}$-2 in $\rca$ is straightforward, because we may uniformly construct the characteristic functions for a sequence of sets defined by a sequence of canonical indices. 

Finally, the implication from $\mathsf{FC}$-2 to $\mathsf{FC}$-3 proceeds by choosing, for each $i$, the least $v_i$ such that $C(i,v_i) = 1$. This gives a function $f$ with $f(i) = v_i$, with the range of $f$ being an infinite enumerated set. By a standard argument, every infinite enumerated set has an infinite (decidable) subset. In this situation, any such subset here will be an infinite totally disconnected set in the original graph. 
\end{proof}

The proof of the next theorem will require a particular technical lemma. Intuitively, this lemma formalizes a standard fact about approximations to $\emptyset'$ that is often required to verify priority arguments: for each $e$ there is a $t$ such that every program with index less than $e$ that halts must halt in no more than~$t$ steps. 
This fact is most directly formalized as a $\Sigma^0_2$ property of~$e$, so some care is necessary to prove the claim in~$\mathsf{RCA}_0$. 

\begin{lemma}\label{lem:tech}
The following is provable in~$\rca$. 
Let $f\colon \mathbb{N}\to\mathbb{N}$ be an injection, and 
$h_s(e) = \max\{n < s : f(n) \leq e\}$ for each $e$ and $s$. Then 
for each $e$ there is a $t$ such that $h_t(e') = h_s(e')$ for all $e' \leq e$ and all $s > t$. 
\end{lemma}
\begin{proof}
We may construct $h$, as a function of $s$ and $e$, using $\Delta^0_1$ comprehension. 
Fix $e$.  By bounded $\Sigma^0_1$ comprehension (Theorem~II.3.9 of Simpson~\cite{simpson}), there is a set $X$ such that, for all $i$, $i \in X$ if and only if $i \leq e$ and $(\exists s)[f(s) = i]$.   Define a function $g \colon \mathbb{N} \to \mathbb{N}$ such that $g(i) = 0$ when $i \not \in X$ and $g(i) = (\mu s) [ f(s) = i]$ for each $i \in X$.  By the strong $\Sigma^0_1$ bounding scheme (a consequence of $\rca$ by Exercise II.3.14 of Simpson~\cite{simpson}),  the image of the interval $[0,e]$ under $g$ is bounded by some number~$t$.  Any such $t$ satisfies the claim.
\end{proof}

The proof of the next theorem can be viewed as a priority argument that has been formalized in $\rca$. 
\newpage

\begin{theorem}\label{thm:finitecomponents}
The following are equivalent over $\rca$:
\begin{itemize}
\item $\aca$.
\item $\mathsf{FC}$-$i$ for $i \in \{1,2,3\}$.
\end{itemize}
\end{theorem}

\begin{proof}
We work in $\mathsf{RCA}_0$. By Lemma~\ref{lem:fcforward}, it is sufficient to assume $\mathsf{FC}$-$3$ and prove that the range of an arbitrary injection $f\colon \mathbb{N}\to\mathbb{N}$ exists. As in 
Lemma~\ref{lem:tech}, $\Delta^0_1$ comprehension is sufficient to define the function 
$h_s(e) = \max\{n < s : f(n) \leq e\}$.  

\proofheader{Construction}
We form a countable graph $G$ with vertex set  $V = \{v_i : i \in \mathbb{N}\}$. We say that $i$ is the \textit{index} of vertex~$v_i$.  We construct a sequence $\langle G_s : s \in \mathbb{N}\rangle$ of finite graphs
such that the vertices of $G_s$ are $\{v_i : i \leq s\}$ and such that if $t > s$ and $v_i,v_j$ are in $G_s$ then edge $(v_i,v_j)$ is in $G_t$ if and only if it is in $G_s$. Then $G = \bigcup_{s \in\nat} G_s$ can be constructed in $\mathsf{RCA}_0$.  To begin, let $G_0$ be a graph with the single vertex $v_0$ and no edges.

We say that a vertex $v_e$ of $G_s$ \textit{requires attention at stage} $s+1$ if there is no $v_i$ in the same component of $G_s$ as $v_e$ with $i < e$ (so $v_e$ has the smallest index of all vertices in its component in $G_s$), and also $h_s(e) > h_t(e)$, where $t$ is maximal such that $v_t$ is in the same component as $v_e$ in $G_s$.  This is a $\Sigma^0_0$ property of $e$ and $s$. If no vertex in $G_s$ requires attention at stage $s+1$, let $G_{s+1}$ be $G_s$ with one additional vertex $v_{s+1}$ and no additional edges.   Otherwise, there is a minimum $e$ such that $v_e$ requires attention at stage $s+1$. In this case, let $G_{s+1}$ be $G_s$ with one new vertex $v_{s+1}$ and an edge between $v_{s+1}$ and $v_i$ when $e \leq i \leq s$.

\proofheader{Verification}
It is straightforward to prove by induction that the vertices of $G_s$ are exactly $\{v_0, \ldots,v_s\}$, and that for all $v_i, v_j$ and all $s$ with $i < j < s$, $v_i$ and $v_j$ are adjacent in  $G_s$ if and only if they are adjacent in~$G_j$.  We will show that each component of $G$ is of the form $\{v_i \in V: m \leq i < n\}$ for some $m < n$.

We first show that each component of $G$ is finite. For this, it is enough to prove that for each component $C$ of $G$ there is an $r$ such that $C$ is a component of $G_r$ (in particular, $C \subseteq G_r$).  Given a component $C$ of $G$, choose $e$ to be minimal such that $v_e \in C$. Choose $t$ as in Lemma~\ref{lem:tech}, so that $h_t(e') = h_s(e')$ for all $e' \leq e$ and all $s > t$.  No vertex $v_{e'}$ with $e' < e$ can require attention 
at or after stage $e$, because otherwise $v_e$ would be in the component of $v_{e'}$ in $G$. Furthermore,
vertex $v_e$ does not require attention at any stage $s > t$, by the choice of~$t$. Thus no new vertices
can be added to the component of $v_e$ after stage $r = \max\{e,t\}$. 

To show that each connected component of each $G_s$ is of the form $\{v_i \in V: m \leq i < n\}$ for some $m < n$, note that all components of $G_0$ are of this form, and the procedure in the construction preserves this property of the graph at each step. Thus, the components of $G$ are either of the desired form or are unbounded. But, by the previous paragraph, every component of $G$ is bounded.

Now assume that $A$ is any infinite totally disconnected set. Arrange $A$ into a sequence $a_0 < a_1 < a_2 < \cdots$ in which the indices of the vertices are strictly increasing.  For each $i$, the component of $a_i$ 
is an interval $\{v_{e(i)}, v_{e(i) + 1}, \ldots, v_{s(i)}\}$. 
The argument in the previous paragraph shows that the component of each vertex $a_i$ in $G$ is the same as the component of $a_i$ in the induced subgraph consisting of the finite set of vertices with index less than that of $a_{i+1}$.  Therefore, the functions $e(i)$ and $s(i)$ can be formed by
$\Delta^0_1$ comprehension using $A$ as a
parameter.  

\textit{Claim:} For all $i$ and $s$, $h_s(e(i)) \leq h_{s(i)}(e(i))$. If not, by the $\Sigma^0_1$ least
element principle, there is a least $i$ such that, for some $s$, $h_s(e(i)) > h_{s(i)}(e(i))$. Choose this least $i$ and then choose a minimal witness $s$. Because $s > s(i)$, $v_s$ is not in the same component of $G$ as $v_{e(i)}$. Thus, by the construction of~$s$, vertex $v_{e(i)}$ requires attention at stage $s+1$, and so the construction causes $v_s$ to be in the same component as $v_{e(i)}$, which is impossible. This proves the claim.

Therefore, using the claim, we see that each $n$ is in the range of $f$ if and only if $n$ is in the range of
$f \upharpoonright \{0, \ldots, s(i)\}$ for the least $i$ such that the index of $a_i$ is greater than $n$. This gives a $\Delta^0_1$ definition of the range of $f$, and so $\mathsf{RCA}_0$ proves that the range exists.
\end{proof}

We obtain two computability-theoretic corollaries from the construction of Theorem~\ref{thm:finitecomponents}.

\begin{corollary}\label{cor:finitecomponents1}
There is a computable graph in which every component is finite such that every infinite totally disconnected set computes $\emptyset'$.
\end{corollary}

\begin{corollary}\label{cor:finitecomponents2}
There is a computable graph in which every component is finite such that there is no infinite c.e. totally disconnected set of vertices.
\end{corollary}

\begin{proof}
Construct the graph from the proof of Theorem~\ref{thm:finitecomponents}. Suppose there is an infinite c.e totally disconnected set of vertices, $A$. Then $A$ has an infinite computable subset, $B$, which is also totally disconnected. Thus, by Corollary~\ref{cor:finitecomponents1}, $B$ computes $\emptyset'$, which is impossible.
\end{proof}


\section{Weak partitions}\label{sec:part}

In this section, we isolate a purely combinatorial principle that underlies the reversal of Theorem~\ref{thm:main}. The principle uses families of sets in the style of Dzhafarov and Mummert~\cite{dm2013}.  In $\rca$, we define an \emph{enumerated sequence of enumerated sets} to be a function $D \colon \mathbb{N} \to \mathbb{N}\times\mathbb{N}$. We abuse notation by writing $D$ as a sequence of sets $\langle D_i : i \in \mathbb{N}\rangle$, and by writing $j \in D_i$ if there is an $s$ with $D(s) = \langle i,j\rangle$. We write $D_i = D_j$ if $(\forall n)[n \in D_i \Leftrightarrow n \in D_j]$, which in turn is an abbreviation for 
\[
(\forall n) [(\exists a)(D(a) = \langle i,n\rangle ) 
	\Leftrightarrow (\exists b)  (D(b) = \langle j,n\rangle ) ].
\]
We may similarly express $D_i \cap D_j = \emptyset$ as a formula with parameter~$D$.
In general, each component is only an enumerated set.  We say that $D_i$ \emph{exists} if there is a set containing exactly those $j \in \nat$ for which there is a $k \in\nat$ with $D(k) = \langle i,j\rangle$.

\begin{defn}
A \emph{weak partition} of~$\nat$  is an enumerated sequence of enumerated sets  $\langle D_i : i \in \mathbb{N}\rangle$ that satisfies the following conditions:
\begin{enumlist}
\item for every $n \in \nat$ there is an $i \in\nat$ with $n \in D_i$, and
\item for all $i,j \in \nat$, either $D_i = D_j$ or $D_i \cap D_j = \emptyset$. 
\end{enumlist}
A weak partition of~$\nat$ is a \emph{strong partition} if it furthermore satisfies the property that $D_i \cap D_j = \emptyset$ when $i\not = j$.  Each enumerated set $D_i$ a \emph{component} of the weak partition. 
\end{defn}

If $\langle D_i : i \in \mathbb{N}\rangle$ is a strong partition of~$\nat$ then $\rca$ proves that the set~$D_j$ exists for each $j \in \nat$. This is because $\rca$ is able to form a function $f\colon \nat \to \nat$ such that $f(m)$ is the unique $j$ such that there exists an~$s$ with $D(s) = \langle j,m\rangle$; then $D_j = f^{-1}\{j\}$ exists by $\Delta^0_1$ comprehension. The situation is different for weak partitions, however.  Rather than providing a direct reversal for the problem of constructing a set in a weak partition of~$\nat$, we show that this problem is intricately tied to the problem of finding a component of a countable graph.

\begin{theorem}\label{thm:family}
The following are equivalent over $\mathsf{RCA}_0$.
\begin{enumlist}
\item Every countable graph has a connected component.
\item For each weak partition $\langle D_i : i \in \mathbb{N}\rangle$ of~$\nat$,  there is an $i$ such that $D_i$ exists.
\end{enumlist}
\end{theorem}

\begin{proof} Working in $\rca$, we first prove that (1) implies (2).  Let $\langle D_i \colon i \in \mathbb{N}\rangle$ be a weak partition of~$\nat$. We build a countable graph $G$ whose vertices are the disjoint union of three infinite sets: $\{d_i : i \in\mathbb{N}\}$, $\{n_i \colon i \in \mathbb{N}\}$, and $\{ a_{i,s} : i,s \in \mathbb{N}\}$.  For all $i,j$, there is no edge between $d_i$ and $d_j$ and no edge between $n_i$ and~$n_j$.
For each~$i$, there is an edge between $d_i$ and $a_{i,0}$, and no edge between $d_i$ and $a_{i,s+1}$ for any~$s$. There is an edge between $a_{i,s}$ and $a_{i,s+1}$ for all $i$ and~$s$, and no other edges among vertices of the form~$a_{i,s}$. There is an edge between $a_{i,s}$ and $n_j$ if and only if $D(s) = \langle i,j\rangle$. A diagram of this construction is shown in Figure~\ref{fig:family}.

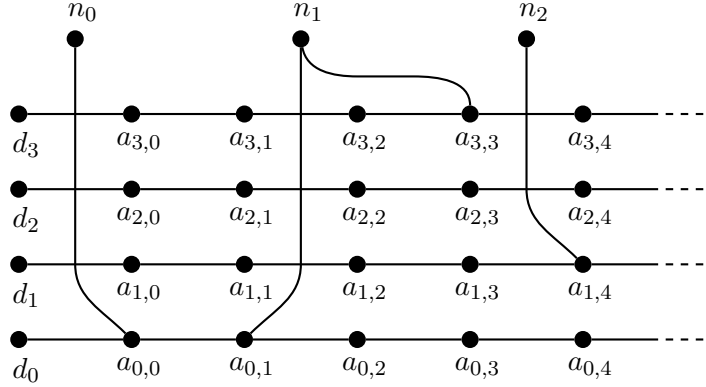
\begin{figure}
\begin{center}

\begin{tikzpicture}[x=1cm,y=1cm,thick]

\tikzstyle{dot}=[draw,circle,inner sep=2pt,fill]
\tikzstyle{labela}=[above,shift={(0.1,0.1)}]
\tikzstyle{labelal}=[above left,shift={(0,0.1)}]
\tikzstyle{labelar}=[above right,shift={(0,0.1)}]

\tikzstyle{labelbl}=[below left,shift={(0,-0.1)}]
\tikzstyle{labelbr}=[below right,shift={(0,-0.1)}]
\tikzstyle{labelb}=[below,shift={(0.1,-0.1)}]

\tikzstyle{labell}=[left,shift={(-0.1,0)}]
\tikzstyle{labelr}=[right]

\newcommand{\biggap}{1}
\newcommand{\smallgap}{1}

\foreach \y in {0,1,2,3} 
{ 
	\foreach \name/\x in { $d_\y$/0, $a_{\y,0}$/1, $a_{\y,1}$/2, $a_{\y,2}$/3, $a_{\y,3}$/4, $a_{\y,4}$/5}
 	{
   		\node[dot] (A\y-\x) at (1.5*\x,\biggap*\y) {};
    		\node[labelb] at  (A\y-\x) {\name};
  	}
  	\foreach \from/\to in {0/1, 1/2,2/3,3/4,4/5}  \draw (A\y-\from) -- (A\y-\to);
}

  \foreach \name/\y in { $n_0$/0, $n_1$/1, $n_{2}$/2}
  {
    \node[dot] (N-\y) at (0.75+3*\y,4*\biggap) {};
      \node[labela] at (N-\y) {\name};
    }
    
  \foreach \y in {0,1,2,3} {
     \draw (A\y-5)--(8.5,\biggap*\y);
     \draw (8.6,\biggap*\y)--(8.7,\biggap*\y);
     \draw (8.8,\biggap*\y)--(8.9,\biggap*\y);
     \draw (9,\biggap*\y)--(9.1,\biggap*\y);
  }


\draw (N-0) to (0.75, 1*\biggap) to[out=-90,in=135] (A0-1);
\draw (N-1) to  (3.75, 1*\biggap) to[out=-90,in=45] (A0-2);

\draw (N-1) to[out=-80,in=180] (4.5,3.5*\biggap)-- (5.25, 3.5*\biggap) to[out=0,in=90] (A3-4);
\draw (N-2)--(6.75,2*\biggap) to[out=-90,in=135] (A1-5);

\end{tikzpicture}
\end{center}
\caption{Illustration of a typical piece of the construction in Theorem~\ref{thm:family}. 
The edges beginning at vertices
of the form $n_i$ indicate that  $0 \in D_0$, $1 \in D_0$, $1 \in D_3$, and $2 \in D_1$.
Because $1 \in D_0 \cap D_3$, it must be that $D_0 = D_3$, so the construction will eventually
add an edge from $n_0$ to some vertex of the form $a_{3,i}$.}\label{fig:family}
\end{figure}

Assuming~ (1), we may choose a connected component $C$ of the graph~$G$. We first claim there is a vertex of the form $d_k$ in $C$.  Let $x$ be any vertex in~$C$. If $x$ is of the form $a_{i,s}$ then $d_i$ is in~$C$. If $x$ is of the form $n_j$ then, by assumption, there is a
$D_i$ such that $j \in D_i$. Let $s$ be such that $D(s) = \langle i,j\rangle$. Then there will be an edge in $G$ from $n_j$ to $a_{i,s}$, and thus $a_{i,s}$ and $d_i$ are in~$C$. This proves the claim. 

Now choose $d_p \in C$ and let $X = \{ j : n_j \in C\}$. We claim that $X = D_p$, which is sufficient to finish the claim that (1) implies (2).  As in the previous paragraph, if $j \in D_p$ then $n_j \in C$. For the converse, assume that $n_j \in C$. Then there is a path $\tau_0$ from $d_p$ to $n_j$ in $G$.   Let $\langle n_{i_0}, n_{i_1}, \dots, n_{i_r} = n_j\rangle$ list the vertices of the  form $n_i$ that are visited by $\tau_0$, in the order they are visited; we will call this sequence the \textit{pattern} of $\tau_0$. 
 If $r= 0$ then it is immediate that $j \in D_p$, because the path must be of the form
\[
d_p \to a_{p,0} \to a_{p,1} \to \cdots  \to a_{p,s} \to n_j
\]
for some~$s$, and then $D(s) = \langle p,j\rangle$. 
Otherwise, if $r > 0$, the matter is more complicated, because we need to 
verify that a path ``directly'' from $d_p$ to $n_j$ can be constructed within $\mathsf{RCA}_0$. We will produce a sequence
$\tau_1, \ldots, \tau_r$ of paths from $d_p$ to $n_j$ such that
the pattern of $\tau_m$ is $\langle n_{i_m}, n_{i_{m+1}},\ldots, n_{i_r}\rangle$. Then the previous argument will apply to $\tau_r$, completing the argument. 

Suppose a path $\sigma$ is given which begins at a vertex $d_p$ and ends at a vertex $n_j$ with pattern $\langle n_{i_0}, n_{i_1}, \dots, n_{i_r} \rangle$. If
$n_{i_0} = n_{i_1}$, then the vertices visited between the first and second visits to $n_{i_0}$ must all be of the form $a_{k,s}$ for some fixed~$k$. 
We can remove $n_{i_0}$ and these vertices from the path, and the result will be a shorter path with the same initial and final vertices as~$\sigma$, 
but with pattern $\langle  n_{i_1}, \dots, n_{i_r} \rangle$.    If $n_{i_0} \not = n_{i_1}$, we look at the vertices visited just before and after $n_{i_0}$; the path will be of the form 
\[
d_p \to \cdots \to a_{p,s_0} \to n_{i_0} \to a_{z,s_1} \to \cdots \to n_{i_1} \to \cdots \to n_{i_r}.
\]
Moreover, all the vertices between $n_{i_0}$ and $n_{i_1}$ will be of the form $a_{z,s}$ or $d_z$. If $z = p$ then we can again remove $n_{i_0}$ by making the first part of the new path proceed directly from $d_p$ to $a_{p,s_1}$ to $n_{i_1}$ and letting the rest be as in~$\sigma$.  If $z \not = p$ then we see
that $D(s_0) = \langle p,n_{i_0}\rangle $ and $D(s_1) = \langle z, n_{i_0}\rangle$,
so $n_{i_0} \in D_p \cap D_z$. Then, because $n_{i_1} \in D_z$, by assumption $n_{i_1}$ is also in $D_{p}$. We may choose any $s'$ with $D(s') = \langle p, n_{i_1}\rangle$ and create a path of the form
\[
d_p \to a_{p,0} \to \cdots \to a_{p,s'} \to n_{i_1} \to \cdots \to n_{i_r}
\]
where the part of the path after $n_{i_1}$ is as in~$\sigma$. 

An analysis of the argument of the previous paragraph shows that $\mathsf{RCA}_0$ proves that, given a path $\sigma$ from a vertex $d_p$ to a vertex $n_j$ with pattern $\langle n_{i_0}, n_{i_1}, \ldots, n_{i_r}\rangle$, with $r > 0$,  there is a path $\sigma'$ which 
has the same beginning and ending vertices and has pattern $\langle n_{i_1}, \ldots, n_{i_r}\rangle$. This relationship between $\sigma$ and $\sigma'$ can be defined by a $\Sigma^0_0$ formula $\phi(\sigma, \sigma')$, so $\mathsf{RCA}_0$ can form the set $Z$ consisting of all such pairs $\langle \sigma,\sigma'\rangle$. Then, by applying $\Sigma^0_1$ induction with $Z$ as a parameter,  $\mathsf{RCA}_0$ proves that, given a 
path $\tau_0$ as above, there is a sequence $\langle \tau_0, \ldots, \tau_r \rangle$ with 
$\phi(\tau_i,\tau_{i+1})$ for all~$i<r$. Then $\mathsf{RCA}_0$ proves that $\tau_r$ is a path from $d_p$ to $n_j$ with pattern $\langle n_j\rangle$, and thus $j \in D_p$, as desired. 

Now, still working in $\rca$, we prove that (2) implies (1). Let $G$ be a countable graph. For each vertex $v_i$, we enumerate a sequence $\langle u_{i,j} : j \in \mathbb{N}\rangle$ by searching (with dovetailing) all possible paths beginning at $v_i$, and  putting $u_{i,j} = w$ if $w$ is the
$j$th vertex that is discovered to be path connected to $v_i$ in this way.
In particular, we may assume that $u_{i,0}= v_i$ for each~$i$.
This construction is uniform in $i$, and so there is an enumerated sequence of enumerated sets $D$ such that $D_i = \{u_{i,j} : j \in \mathbb{N}\}$ for each $i$ (of course $D_i$ may not exist as a set). 

We need to verify that $D$ is a weak partition of~$\nat$.  This follows from the facts that $v_i \in D_i$ for each~$i$ and that $\mathsf{RCA}_0$ can concatenate paths in the graph. Thus, if two vertices $v$ and $w$ are both path connected to a vertex $u$, and $v$ is also path connected to a vertex $u'$, then $w$ is also path connected to~$u'$. 

Now we may apply~(2) to obtain a set $X$ and index $i$ with $X = D_i$. It is immediate that $X$ is the connected component of vertex $v_i$ in~$G$, as desired.
\end{proof}

Applying Theorem~\ref{thm:main}, we obtain a corollary on the reverse mathematics of weak partitions. We also obtain a purely computability theoretic corollary by re-examining the proof of Theorem~\ref{thm:family} in light of Corollary~\ref{cor:main}.

\begin{corollary} 
The principle that if 
$\langle D_i : i \in \mathbb{N}\rangle$ is a weak partition of~$\nat$ then $D_i$ exists for some $i \in \nat$ is equivalent to 
$\mathsf{ACA}_0$ over~$\mathsf{RCA}_0$.
\end{corollary}

\begin{corollary}
There is a computable weak partition $\langle D_i : i \in \mathbb{N}\rangle$ of $\nat$  \textup{(}that is, the function $D\colon \nat \to \nat \times \nat$ is computable\textup{)} such that for each $i \in\nat$ the set $D_i$ computes \textup{(}and is thus Turing equivalent to\textup{)}~$\emptyset'$.
\end{corollary}

\section{Parallelization and relationships with $\mathsf{LPO}$}

In the previous sections, we have explored the problem of finding a component of a graph using
the methods of reverse mathematics. In this section, we study the problem from the viewpoint of Weihrauch reducibility, as described by Brattka and Gherardi~\cites{bg2011, bg2011b} and by Dorais, Dzhafarov, Hirst, Mileti, and Shafer~\cite{ddhms2013}.  
From this viewpoint, each set $\wprob{C} \subseteq 2^\omega \times 2^\omega$ can be viewed a \emph{Weihrauch problem} in the following way.
Each $A \in 2^\omega$ is viewed as an \emph{instance} of the problem.
Given an instance $A$, a corresponding \emph{solution} is a $B \in 2^\omega$ such that $(A,B) \in \wprob{C}$. The problems studied in this paper are given by $\Pi^1_2$ sets, and therefore every instance has at least one solution. 

Following Definition~1.5(2) of Dorais \textit{et al.}, we say that a Weihrauch problem~$\wprob{Q} \subseteq 2^\omega \times 2^\omega$ is \emph{strongly Weihrauch reducible} to a Weihrauch problem $\wprob{R} \subseteq 2^\omega \times 2^\omega$  if there are Turing reductions $\Phi$ and $\Psi$ such that whenever $A$ is an instance of $\wprob{Q}$, $B = \Phi(A)$ is an instance of $\wprob{R}$ such that whenever $T$ is a solution to $B$, $S = \Psi(T)$ is a solution to~$A$. This relation is denoted $\wprob{Q}\leq_{\sW} \wprob{R}$. Two problems are \textit{strongly Weihrauch equivalent} if each is strongly Weihrauch reducible to the other.

We have claimed that the proof of Theorem~\ref{thm:family} provides a tight link between the problem of computing a connected component of a countable graph and the problem of computing a component of a weak partition of~$\nat$.  This link cannot be established directly through methods of reverse mathematics, however, because the relation of proof theoretic equivalence in reverse mathematics is too coarse. Strong Weihrauch equivalence can be used to establish a tighter link between principles that emphasizes the uniformity of the relationship between the principles. For example, the following corollary follows from an analysis of the proof of Theorem~\ref{thm:family}.

\begin{corollary} The problem of finding a connected component of a countable graph is strongly Weihrauch equivalent to the problem of finding a component of a weak partition of~$\nat$.
\end{corollary}

In this section, we study two particular Weihrauch problems.
\begin{itemize}
\item $\wprob{P}$: the Weihrauch problem of finding a connected component of a countable graph. 
\item $\wprob{D}$: the Weihrauch problem of decomposing a countable graph into connected components.  As in Section~3, a decomposition is
a map from vertices to $\nat$ such that $f(v_0) = f(v_1)$ if and only if $v_0$ and $v_1$ are in the same connected component. 
\end{itemize}
We show that $\wprob{P}$ and $\wprob{D}$ are equivalent to their ``parallelized'' forms, and we obtain results that locate these principles relative to better-known Weihrauch degrees (Theorems~\ref{wr3} and~\ref{wr4}).  Many results in this section capitalize on the uniformity of the constructions in the previous sections.   The authors thank one of the referees for suggesting the results of Theorems~\ref{wr4} and~\ref{wr5} and the inclusion of $\mathsf{LPO}$ in Theorem~\ref{wr3}.

We begin by considering the \emph{binary parallelization} $\langle \wprob{P},\wprob{P}\rangle$.  An instance of this problem consists of a pair of countable graphs $(G_1,G_2)$; a solution is a pair of sets $(C_1,C_2)$ such that $C_1$ is a connected component of $G_1$ and $C_2$ is a connected component of $G_2$.  In the terminology more common in the literature on Weihrauch reducibility, the binary parallelization is simply the product of the problem $\mathsf{P}$ with itself. Additional information on parallelization and the product operation is given by Brattka and~Gherardi~\cites{bg2011,bg2011b},
by Dorais~\textit{et al.}~\cite[Section~2]{ddhms2013}, and by Pauly~\cite{pauly2010}.

\begin{theorem}\label{wr1} $\wprob{P}$ is strongly Weihrauch equivalent to its binary parallelization.
\end{theorem}
\begin{proof}
It is immediate that $\wprob{P}$, like every other problem, is Weihrauch reducible to its parallelization $\langle \wprob{P},\wprob{P}\rangle$. We may effectively convert any computable graph $G$ into an instance of $\langle \wprob{P},\wprob{P}\rangle$ by setting $G_1 = G_2 = G$. Then, given a solution $(C_1,C_2)$, we simply take the first element of the pair,~$C_1$.

The proof of the converse is  conceptually similar to the 
proof that the computability problem corresponding to weak K\"onig's lemma is Weihrauch equivalent to its binary parallelization.
Suppose that $G_1 = (V_1, E_1)$ and $G_2 = (V_2,E_2)$ are countable graphs. 
Form a new graph $G_1\times G_2$ in the following way.  The vertices of $G_1 \times G_2$ consist of all pairs $(u,v)$ where $u \in V_1$ and $v \in V_2$. There is an edge between $(u,v)$ and $(x,y)$ if and only if either $u = x$ and $(v,y) \in E_2$, or $v = y$ and $(u,x) \in E_1$. Note that if $(u,x) \in E_1$ and $(v,y) \in E_2$ then there is a path from $(u,v)$ to $(x,y)$ in $G_1 \times G_2$. 

Suppose that $C$ is a connected component of $G_1 \times G_2$ and $(u,v)$ is a vertex in $C$.  We claim that a vertex $(u,y)$ is in $C$ if and only if $v$ and $y$ are in the same component of $G_2$. For the forward direction, given a path from $(u,v)$ to $(u,y)$ in $G_1 \times G_2$, we can ignore the first components of the vertices, remove any duplicates, and get a path from $v$ to $y$ in $G_2$. Conversely, given a path from $v$ to $y$ in $G_2$, we can make an ordered pair with $u$ and each vertex of the path to form a path from $(u,v)$ to $(u,y)$ in~$G_1\times G_2$.   This proves the claim.  Thus we may compute a connected component of $G_2$ from any connected component of $G_1\times G_2$. The proof that we may compute a component of $G_1$ is similar.
\end{proof}

According to Definition 2.2 of Dorais \textit{et al.}~\cite{ddhms2013}, a $\Pi^1_2$ problem $\wprob{R}$ has {\emph{finite tolerance}} if there is a Turing functional 
$\Theta$ such that whenever $B_1$ and $B_2$ are instances of $\wprob{R}$ with $B_1(n) = B_2 (n)$ for all $n \ge m$, and $S$ is a 
solution of $B_1$, then $\Theta (S,m)$ is a solution to $B_2$.  If the graph component problem $\wprob{P}$ had finite tolerance, we could use our Theorem~\ref{wr1} above and
Theorem~2.5 of Dorais \textit{et al.}~\cite{ddhms2013} to show that the infinite parallelization $\widehat{\wprob{P}}$ (defined below) is strongly Weihrauch reducible to~$\wprob{P}$.
The next theorem rules out this proof strategy.

\begin{theorem}\label{wr2}
The problem of finding a connected component of a countable graph does not have finite tolerance.
\end{theorem}

\begin{proof}
We construct a computable graph $G$ with vertex set $V = \{ v_i : i \in \nat\}$. Such a graph may be presented as the 
characteristic function of its set of edges, using an indexing $\{e_i : i \in \nat\}$ of the edges of the complete graph on~$V$. Without loss of generality, we may assume that
the first potential edge $e_0$ would connect $v_0$ to $v_1$.  Let $B_1$ be the graph consisting of a complete graph the on vertices $\{ v_{2i} : i \in \nat\}$ 
and a complete graph on the vertices $\{v_{2i+1} : i \in \nat\}$.  Let $B_2$ be $B_1$ with the edge $e_0$ added.  Let $S$ be the solution to $B_1$ consisting
of the even vertices.  Suppose by way of contradiction that $\Theta$ is a Turing functional witnessing that $P$ has finite tolerance.
Because $B_1(n)=B_1(n)$ for all $n\ge 1$, $\Theta(S, 1)$ is a solution of $B_1$, that is, a connected component of $B_1$.
Because $B_1(n)=B_2(n)$ for all $n\ge 1$, $\Theta(S, 1)$ is also a connected component of $B_2$,
contradicting the fact that $B_1$ and $B_2$ have no common connected components.
\end{proof}

Despite the inapplicability of Theorem~2.5 of Dorais \textit{et al.}~\cite{ddhms2013}, we can directly prove that the \emph{infinite parallelization} $\widehat{\wprob{P}}$ is strongly Weihrauch equivalent to $\wprob{P}$. An instance of the parallelization $\widehat{\wprob{P}}$ consists of an infinite sequence $\langle G_i \rangle_{i\in\nat}$ of countable graphs; a solution is a sequence $\langle C_i\rangle_{i\in\nat}$ of sets such that $C_i$ is a component of $G_i$ for each~$i \in \mathbb{N}$.  An infinite parallelization can also be seen as an infinite product of the problem $\wprob{P}$ with itself.
%

Our approach will capitalize on established results from the theory of Weihrauch degrees.  The following
two facts are included in Proposition~4.2 of Brattka and~Gherardi~\cite{bg2011b}.  First, Weihrauch reducibility
is preserved by parallelization, so $\wprob{Q} \leq_{\sW} \wprob{R}$ implies $\widehat{\wprob{Q}}\leq_{\sW} \widehat {\wprob{R}}$.  Second,
parallelization is idempotent, so $\widehat {\wprob{Q}} \equiv_{\sW} \widehat{\widehat{\wprob{Q}\,}}$.  

We will make use
of the \emph{limited principle of omniscience}, denoted $\wprob{LPO}$, which asserts that if
$p\colon \mathbb{N} \to \mathbb{N}$ then either there exists an $n \in \mathbb{N}$ such that $p(n)=0$, or
$p(n) \neq 0$ for all $n \in \mathbb N$. Although this principle is trivially provable in classical logic, it has nontrivial Weihrauch degree, as the nonuniform classical proof suggests.  The infinite parallelization $\widehat{\mathsf{LPO}}$ is much stronger in reverse mathematics. When stated as a $\Pi^1_2$ principle, $\widehat{\mathsf{LPO}}$ is equivalent to $\mathsf{ACA}_0$ over~$\mathsf{RCA}_0$.

\begin{theorem}\label{wr3}
Each of the following is strongly Weihrauch equivalent to the others:
$\wprob{P}$, $\widehat{\wprob{P}}$, $\wprob{D}$, $\widehat{\wprob{D}}$, and $\widehat{\wprob{LPO}}$.
\end{theorem}

\begin{proof}
The proof of the theorem will follow easily once we establish the following strong Weihrauch reductions:
\[
\widehat{\wprob{LPO}} \leq_{\sW} \wprob{P} \leq_{\sW} \wprob{D} \leq_{\sW} \widehat{\wprob{LPO}}.
\]
Let $\langle p_i \rangle_{i \in \mathbb N}$ be a sequence of functions viewed as an instance of $\widehat{\wprob{LPO}}$.
We will conflate ordered pairs with their integer codes and assume the pairing function is bijective.
Define
\[
f((i,j)) = 
\begin{cases}
2i+1 & \text{if}~ p_i (j) = 0 \text{ and }  (\forall t<j) [ p_i (t) \neq 0 ],\\
2(i,j) & \text{otherwise.}
\end{cases}
\]
Then $f$ is an injection, and each odd number $2i+1$ is in the range of $f$ if and only if $(\exists n)[p_i (n) = 0]$.
Use the construction from the proof of Theorem~\ref{thm:main} to build a graph $G$; it can be seen that this construction is uniform.  As  in that theorem, each connected component of $G$ computes (in a uniform manner) the range of $f$, so $\{i : ( \exists n)[p_i (n) = 0]\} = \{ i : (\exists n)[f(n) = 2i+1]\}$ is a solution to the instance of $\widehat{\wprob{LPO}}$ computable from the connected component.  Thus $\widehat{\wprob{LPO}} \leq_\sW \wprob{P}$.

If $G$ is a graph with vertex $v$, each decomposition of $G$ into connected components computes the connected
component containing~$v$.  Thus $\wprob{P} \leq_{\sW} \wprob{D}$.  It remains to show that
$\wprob{D} \leq_{\sW} \widehat{\wprob{LPO}}$.  Suppose $G$ is a graph with vertices $\{v_i : i \in \mathbb{N} \}$.
For each pair $(i,j)$, let $p_{(i,j)}(n)=0$ if $n$ is a code for a path from $v_i$ to $v_j$ and let $p_{(i,j)}(n) = 1$
otherwise.  Applying $\widehat{\wprob{LPO}}$ yields the set $S = \{ (i,j) : \exists n (p_{(i,j)} (n) = 0)\}$, which is
the set of codes for all pairs of path connected vertices of $G$.
The function $d(v_n) =( \mu\,  i < n) [(i,n) \in S]$ (where the function returns $n$ if the finite search fails)
is a decomposition of $G$ into connected components that is computable from $S$.  Thus
$\wprob{D} \leq_{\sW} \widehat{\wprob{LPO}}$, completing our initial chain of reductions.

By transitivity of Weihrauch reducibility, we may conclude that
$\wprob{P} \equiv_{\sW} \wprob{D} \equiv_{\sW} \widehat{\wprob{LPO}}$.  Because parallelization is idempotent and preserves Weihrauch reducibility, we have 
$
\widehat{\wprob{P}} \equiv_{\sW} \widehat{\wprob{D}} \equiv_{\sW}\widehat{ \widehat{\wprob{LPO}}}\equiv_{\sW} \widehat{\wprob{LPO}},
$
completing the proof of the theorem.
\end{proof}

The techniques of Theorem~\ref{wr3} can also be applied to the properties ${\sf{FC}}\text{-}i$ defined in Section~\ref{section4}.

\begin{theorem}\label{wr4} For $i \in \{ 1, 2, 3\}$, the principles ${\sf{FC}}\text{-}i$ are strongly Weihrauch equivalent to $\widehat{\wprob{LPO}}$, as are their infinite parallelizations.
\end{theorem}

\begin{proof}
As in the proof of the preceding theorem, we first prove a chain of Weihrauch reductions:
\[
\widehat{\mathsf{LPO}} \leq_{\sW}
{\sf FC}{\text -3}  \leq_{\sW} {\sf FC}{\text -2} \leq_{\sW} {\sf FC}{\text -1}
\leq_{\sW} \widehat{\mathsf{LPO}}.
\]

To prove that $\widehat{\wprob{LPO}}\leq_{\sW} {\sf FC}\text{-3}$, fix an instance of $\widehat{\wprob{LPO}}$
and construct an associated injection $f$ as in the proof of Theorem~\ref{wr3}.  Apply the (uniform) construction from
the proof of Theorem~\ref{thm:finitecomponents} to $f$ to obtain a graph in which every component is finite.
As in that proof, any infinite totally disconnected set computes the range of $f$, and so computes a solution
to the instance of~$\widehat{\wprob{LPO}}$.

The proof of Lemma~\ref{lem:fcforward} shows that
${\sf FC}{\text -3}  \leq_{\sW} {\sf FC}{\text -2} \leq_{\sW} {\sf FC}{\text -1}$.  It remains to prove
that $ {\sf FC}{\text -1}  \leq_{\sW} \widehat{\wprob{LPO}}$.  Let $G$ be an infinite graph in which every
connected component is finite.  Construct an instance of $ \widehat{\wprob{LPO}}$ consisting of a sequence
of functions $p_i\colon \mathbb N \to \mathbb N$,
defined as follows.
Set $p_i (n) = 1$ if  (1) $i$ is a canonical index
for $F$, a finite collection of vertices, (2) $G$ restricted to $F$ is a connected subgraph, and (3) if $v_n\notin F$,
then there is no edge between $v_n$ and any vertex listed by~$F$.  Otherwise, set $p_i (n) = 0$.
Note that $i$ is the canonical index of a connected
component of $G$ if and only if $p_i(n) \neq 0$ for all $n$.  Consequently, any solution of the instance of
$ \widehat{\wprob{LPO}}$ computes the sequence of codes for connected components satisfying~${\sf FC}{\text -1} $.

Because parallelization is idempotent and preserves Weihrauch reducibility, the parallelizations of $ {\sf FC}{\text -1}$,
$ {\sf FC}{\text -2}$, and $ {\sf FC}{\text -3}$ are also Weihrauch equivalent to $ \widehat{\wprob{LPO}}$.
\end{proof}

The Weihrauch degree of $\widehat{\wprob{LPO}}$ is the same as that of $\wprob{lim}$, which has been widely
studied.  For example, it is Weihrauch complete for effective $\Sigma^0_2$ measurable functions
\cite{b2005}, and equivalent to computing
the Radon-Nikodym derivative \cite{h2012},
constructing a measure with a given closed set as support \cite{pf2014},
and computing a subgame perfect equilibrium of an infinite sequential game
with open/closed payoffs \cite{lp2014}.

Weihrauch reducibility can also be used to analyze statements about connected components and decompositions of graphs with finitely many connected components.  We let ${\sf P}_k$ be the problem of finding a a connected component of a countable graph that has exactly $k$ connected components,
and let ${\sf D}_k$ be the problem of decomposing a graph with exactly $k$ connected components into its connected components.  These graphs do not come with any witness for the number of components; for example $\mathsf{D}_5$ is simply the restriction of $\mathsf{D}$ to graphs with exactly~$5$ components.

The next theorem relates $\mathsf{D}_k$ and $\mathsf{P}_k$ to $\sf C_{\nat}$, the Weihrauch principle for closed choice on the natural numbers.  The principle $\sf C_{\nat}$ asserts that, given a non-surjective map $p\colon \nat \to \nat$, we can find an $n \in \nat$ which is not in the range of $p$.  The motivation for the name is that the complement of the range of $p$ can be viewed as a closed subset of $\nat$.  For more information about closed choice, see Brattka, de Brecht, and Pauly \cite{bbp2012}.
\newpage

\begin{theorem}\label{wr5}
For each $k \ge 2$, ${\sf P}_k \equiv_\sW {\sf D}_k \equiv_\sW\sf C_\nat$.
\end{theorem}

\begin{proof}
First we will show that $\sf C_\nat \le_\sW P_2$.  Let $p\colon \nat \to \nat$ enumerate the complement of
a nonempty subset of $\nat$.  Construct a related graph $G$ as follows.  Let
$V = \{ u_i : i \in \nat\}\cup\{v_i : i \in \nat\}$ be the vertices of $G$ and include all edges of the
form $(u_i , u_{i+1})$ in $G$.  Consider the finite (possibly empty) increasing sequence $b_i$
defined as follows.  Let $b_0 = (\mu\, t)[ p(t) = 0 ]$ if such a $t$ exists.  Let
$b_{i+1} = (\mu\, t)[ t > b_i \land (\exists s \le t)( p(s) = i+1)]$ if such a $t$ exists.  Note that
there is a computable procedure that determines for each $j \in \nat$ whether $j$ is equal to some~$b_i$.  For each $n$, if $n=b_i$, then add the edge $(u_n , v_n)$ to~$G$;  otherwise,
add $(v_n , v_{n+1} )$ to $G$.  The resulting graph will be similar to the subgraph $G_{\langle 0 \rangle}$
shown in Figure~\ref{fig:finite2}, with the number of caps equal to the least $n$ not in the range of $p$.  In particular, the resulting graph $G$ will have exactly two connected components, one of which will contain all vertices of the form~$u_i$.

Suppose $C$ is a connected component of $G$. 
 If $u_0 \in C$, some (least) $v_n$ must not be in $C$.
Locate it and call it $v_{n_0}$.  If $u_0 \notin C$, some (least) $v_n$ must be in $C$.  Locate it and call
it $v_{n_0}$.  In either case, the number of edges of the form $(u_j, v_j )$ with $j \le n_0$ is the least integer not in the range of $p$.  Thus $\sf C_\nat \le_\sW {\sf P}_2$.
 
To see that $\sf C_\nat \le _\sW {\sf P}_k$ for $k >2$, construct $k-1$ copies of the graph in the preceding
paragraph, using vertices $V_j = \{ u^j_i : i \in \nat \} \cup \{ v^j_i : i \in \nat \}$ for each $j\le k-1$.
For each positive $j \le k-1$, add the edge $(u_0^0 , u_0^j)$.  By an argument similar to the
preceding paragraph, any connected component of the graph can be used to locate an integer
not in the range of $p$.

If $G$ is a graph with vertex $v$, then from any decomposition of $G$ into connected components we can compute the component containing $v$.  Thus, for each $k$, ${\sf P}_k \le_\sW {\sf D}_k$.  It remains
to show that ${\sf D}_k \le_\sW \sf C_\nat$.
Let $G$ be a graph with
exactly $k$ connected components.  Let $\{s_i: i \in \nat\}$ be a bijective enumeration of the size $k$ subsets of the vertices
of~$G$.  Let $\{t_j:j\in \nat\}$ be an enumeration of all the finite paths in $G$.  We will identify pairs $(i,j)$ with their
integer codes, assuming a bijective encoding.  Define $p(i,j) = i+1$ if $t_j$ includes a path between two
vertices of $s_i$, and $p(i,j) = 0$ otherwise.  Since $G$ has exactly $k$ components, some $s_i$ has no
connected vertices, and consequently $i+1$ is not in the range of $p$.  Thus $p$ is non-surjective.
If $i+1$ is any integer not in the range of $p$, then every vertex of $G$ is either an element of $s_i$ or
connected by a path to exactly one element of $s_i$.  Using $s_i$, we can compute a decomposition of
$G$ into connected components.  Thus ${\sf D}_k \le_\sW \sf C_\nat$.
The theorem follows by transitivity of Weihrauch reducibility.
\end{proof}

\bibliographystyle{amsplain}
\bibliography{hmg}
\end{document}